\documentclass{amsart}

\usepackage[dvipsnames,svgnames,x11names,hyperref]{xcolor}
\usepackage{textgreek,bbm,url,graphicx,verbatim,amssymb,enumerate,stmaryrd,booktabs,lmodern,mathtools,microtype,mathabx}
\usepackage[pagebackref,colorlinks,citecolor=Mahogany,linkcolor=Mahogany,urlcolor=Mahogany,filecolor=Mahogany]{hyperref}
\usepackage[capitalize]{cleveref}
\usepackage[mathscr]{euscript}
\usepackage[margin=1.5in]{geometry}
\usepackage{tikz,tikz-cd}
\usetikzlibrary{matrix,calc,positioning,arrows,decorations.pathreplacing,patterns,arrows,decorations.markings,bending}
\tikzset{->-/.style={decoration={
			markings,
			mark=at position #1 with {\arrow{>}}},postaction={decorate}}}

\newtheorem{theorem}{Theorem}[section]
\newtheorem{lemma}[theorem]{Lemma}
\newtheorem{proposition}[theorem]{Proposition}
\newtheorem{corollary}[theorem]{Corollary}
\newtheorem*{corollary*}{Corollary}

\newtheorem{atheorem}{Theorem}

\theoremstyle{definition}

\newtheorem{convention}[theorem]{Convention}

\theoremstyle{remark}

\newtheorem{remark}[theorem]{Remark}
\newtheorem*{remark*}{Remark}

\renewcommand{\rm}[1]{{\mathrm{#1}}}
\newcommand{\lra}{\longrightarrow}

\newcommand{\bb}[1]{{\mathbb{#1}}}

\newcommand{\id}{\rm{id}}

\newcommand{\Top}{\rm{Top}}
\newcommand{\OO}{\rm{O}}
\newcommand{\GG}{\rm{G}}
\newcommand{\PL}{\rm{PL}}
\newcommand{\Sq}{\rm{Sq}}

\calclayout

\newcommand{\circled}[1]{\raisebox{.5pt}{\textcircled{\raisebox{-.9pt} {#1}}}}

\DeclareFontFamily{U}{min}{}
\DeclareFontShape{U}{min}{m}{n}{<-> udmj30}{}

\title{The first two $k$-invariants of $\Top/\OO$}

\author{Alexander Kupers}
\address{Department of Computer and Mathematical Sciences \\ 
	University of Toronto Scarborough \\
	1265 Military Trail, Toronto, ON M1C 1A4 \\
	Canada}
\email{a.kupers@utoronto.ca}

\begin{document}

\begin{abstract}We show that the first two $k$-invariants of $\Top/\OO$ vanish and give some applications.\end{abstract}

\maketitle 

In this note we investigate the $k$-invariants of the infinite loop space $\Top/\OO$, which plays an important role in smoothing theory of topological manifolds. Its homotopy groups $\pi_k(\Top/\OO)$ are given by $\bb{Z}/2$ if $k = 3$ (corresponding to the Kirby--Siebenmann invariant) and the group $\Theta_k$ of oriented homotopy $k$-spheres otherwise (determined by Kervaire--Milnor). We will prove that its first two $k$-invariants vanish, which amounts to:

\begin{atheorem}\label{athm:main} There exists a $9$-connected map
	\[\Top/\OO \lra K(\bb{Z}/2,3) \times K(\bb{Z}/28,7) \times K(\bb{Z}/2,8).\]
\end{atheorem}

\begin{remark*}\,
	\begin{itemize} 
		\item It remains an open question whether there is a preferred choice of map as in \cref{athm:main}.
		\item It remains an open question whether there is an equivalence of spectra $\rm{top/o} \simeq \Sigma^3 H\bb{Z}/2 \vee \Sigma^7 H\bb{Z}/28 \vee \Sigma^8 H\bb{Z}$.
		\item It is a direct consequence of \cref{athm:main} that the map
		\[[M,\Top/\OO] \lra [M,K(\bb{Z}/2,3) \times K(\bb{Z}/28,7) \times K(\bb{Z}/2,8)]\]
		is a bijection if $\dim(M) \leq 8$, and a surjection if $\dim(M) = 9$. Since the map
		\begin{equation}\label{eqn:maps}[M,\Top/\OO] \lra [M,\Top/\PL]\end{equation}
		agrees under this bijection with projection to $[M,K(\bb{Z}/2,3)]$, it is surjective if dimension $\dim(M) \leq 9$. Through smoothing theory, this yields a rephrasing of \cite[Corollary 6.3]{MoritaSmooth}, \emph{``smoothability of 8 and 9 PL manifolds is topologically invariant''}, which means that if two PL-manifolds $M_\alpha$ and $M_\beta$ are homeomorphic then $M_\alpha$ admits a compatible smooth structure if and only if $M_\beta$ does.
		\item On the other hand \cite[Theorem 6.1]{MoritaSmooth} says that there are smooth manifolds of dimension $\geq 22$ with a PL-structure on their underlying topological manifolds that does not admit a compatible smooth structure. In fact, Morita's proof uses that the 10-skeleton of this $22$-dimensional manifold can not be smoothed. Applying the above reasoning, this implies that the next $k$-invariant of $\Top/\OO$ must be non-zero.
		\item \cref{athm:main} is used in \cite[Addendum 1.2]{BustamanteTshishiku}, which concerns the classification of free,
		orientation-preserving actions of finite groups on 7-dimensional aspherical space forms. Indeed, a question in the previous version of this paper inspired us to prove \cref{athm:main}.
	\end{itemize}
\end{remark*}

\subsection*{Acknowledgements} AK thanks Mauricio Bustamante, Darmuid Crowley, Ian Hambleton, and Bena Tshishiku for helpful conversations and comments on earlier versions of this note, as well as the anonymous referee for several improvements to the arguments. AK acknowledges the support of the Natural Sciences and Engineering Research Council of Canada (NSERC) [funding reference number 512156 and 512250]. AK is supported by an Alfred P.~Sloan Research Fellowship. 

\section{Preliminaries and strategy} The space $\Top/\OO$ is 1-connected and has finite homotopy groups, so as a consequence of the arithmetic fracture square it is the product of its $p$-localisations. Moreover, in the range of interest for \cref{athm:main} the $k$-invariants are 2-local so it suffices to work 2-locally, that is, we invert all primes \emph{except} 2.

\begin{convention}In the remainder of this note all spaces, spectra, and maps will be 2-localised.\end{convention}

There is a fibre sequence of infinite loop spaces $\PL/\OO \to \Top/\OO \to \Top/\PL \simeq K(\bb{Z}/2,3)$, corresponding to a fibre sequence of connective spectra 
\[\rm{pl/o} \lra \rm{top/o} \lra \rm{top/pl} \simeq \Sigma^3 H\bb{Z}/2,\] so we can recover $\rm{top/o}$ as the fibre of $\rm{top/pl} \simeq \Sigma^3 H\bb{Z}/2 \to \Sigma\, \rm{pl/o}$. As the Postnikov truncation $\tau_{\leq 8} \rm{pl/o}$ is equivalent to a wedge sum of Eilenberg--Mac Lane spectra $\Sigma^7 H\bb{Z}/4 \vee \Sigma^8 H\bb{Z}/2$ \cite{Jahren,BKM}, we have maps of infinite loop spaces
\[\Top/\OO \lra K(\bb{Z}/2,3) \lra K(\bb{Z}/4,8) \times K(\bb{Z}/2,9)\]
so that the map from $\Top/\OO$ to the fibre of the right map is $9$-connected. The right map is determined by a pair of cohomology classes $k_1 \in H^8(K(\bb{Z}/2,3);\bb{Z}/4)$ and $k_2 \in H^9(K(\bb{Z}/2,3);\bb{Z}/2)$. \cref{athm:main} would follow from:

\begin{atheorem}\label{athm:k} Both $k_1$ and $k_2$ are zero.
\end{atheorem}

\begin{remark}
The class $k_1 \in H^8(K(\bb{Z}/2,3);\bb{Z}/4)$ is the first $k$-invariant of $\Top/\OO$. Given that it vanishes, the second $k$-invariant of $\Top/\OO$ lies in $H^9(K(\bb{Z}/2,3) \times K(\bb{Z}/4,8);\bb{Z}/2)$, which is equal to $H^9(K(\bb{Z}/2,3);\bb{Z}/2) \oplus H^9(K(\bb{Z}/4,8);\bb{Z}/2)$ by the K\"unneth theorem, and with respect to this decomposition the second $k$-invariant is given by $(k_2,0)$.
\end{remark}

Our strategy is to start with the pullback square of infinite loop spaces
\[\begin{tikzcd} \Top/\OO \rar \dar & \GG/\OO \dar \\
	\Top/\PL \rar & \GG/\PL.\end{tikzcd}\]
It will follow from the work of Madsen and Milgram that the horizontal maps are surjective on $\pi_3$, so upon passing to $2$-connective covers we get another pullback square of infinite loop spaces
\[\begin{tikzcd} \Top/\OO \rar \dar & \GG/\OO\langle 2 \rangle \dar \\
	\Top/\PL \rar & \GG/\PL\langle 2 \rangle.\end{tikzcd}\]
This implies that the classes $k_1$ and $k_2$ in the cohomology of $\Top/\PL$ are pulled back from the cohomology of $\GG/\PL\langle 2 \rangle$. In the \cref{sec:madsenmilgram} we will understand the bottom map, in \cref{sec:cohomology} we will compute some relevant cohomology groups, and in \cref{sec:proofs} we will combine these results to prove \cref{athm:k}.

\section{The work of Madsen--Milgram} \label{sec:madsenmilgram}
There is a fibre sequence
\[\Top/\PL \lra \GG/\PL \lra \GG/\Top\]
whose right map (2-locally, per convention) admits a simple description. Let 
\[E \coloneqq  K(\bb{Z}/2,2) \times_{\delta \Sq^2 \iota_2} K(\bb{Z},4)\]
denote the 2-stage Postnikov system with $k$-invariant given by $\beta \Sq^2 \iota_2 \in H^5(K(\bb{Z}/2,2);\bb{Z})$ with $\beta \colon H^4(K(\bb{Z}/2,2);\bb{Z}/2) \to H^5(K(\bb{Z}/2,2);\bb{Z})$ the Bockstein for the short exact sequence of coefficients $0 \to \bb{Z} \to \bb{Z} \to \bb{Z}/2 \to 0$ (c.f.~\cite[Theorem 4.8]{MadsenMilgram}).

\begin{proposition}The map $p \colon \GG/\PL \to \GG/\Top$ is equivalent to the product of a map
	\[f \colon E \lra K(\bb{Z}/2,2) \times K(\bb{Z},4)\]
	with the identity map of $\prod_{k \geq 2} K(\bb{Z}/2,4k+2) \times \prod_{\ell \geq 2} K(\bb{Z},4\ell)$.
\end{proposition}

\begin{proof}There is a map of fibre sequences of infinite loop spaces
	\[\begin{tikzcd} \tau_{\geq 5} \GG/\PL \rar \dar{\circled{1}} & \GG/\PL \dar \rar & \tau_{\leq 4} \GG/\PL \dar{\circled{2}} \\
	\tau_{\geq 5} \GG/\Top \rar & \GG/\Top \rar & \tau_{\leq 4} \GG/\Top\end{tikzcd}\]
	By \cite[Theorem 7.1]{MadsenMilgram}, we can identify the bottom row as double loop spaces with the split fibre sequence of Eilenberg--Mac Lane spaces
	\[\prod_{k \geq 2} K(\bb{Z}/2;4k+2) \times \prod_{\ell \geq 2} K(\bb{Z},4\ell) \lra \prod_{k \geq 1} K(\bb{Z}/2;4k+2) \times \prod_{\ell \geq 1} K(\bb{Z},4\ell) \lra K(\bb{Z}/2,2) \times K(\bb{Z},4).\] 
	Using the equivalence $\Top/\PL \simeq K(\bb{Z}/2,3)$ the map $\circled{1}$ is an equivalence of infinite loop spaces, the top row is also split and the middle map can be written as a product $\circled{1} \times \circled{2}$. We have already identified $\circled{1}$ with the identity map of a product of Eilenberg--Mac Lane spaces and by \cite[Theorem 7.25]{MadsenMilgram} the map $\circled{2}$ can be identified with a map
	\[E \lra K(\bb{Z}/2) \times K(\bb{Z},4).\qedhere\]
\end{proof}

From the long exact sequence of homotopy groups and the identification of the fibre of $p$ with $\Top/\PL$, we see that on homotopy groups the composition 
\[Y \lra \GG/\PL \overset{p}\lra \GG/\Top \lra K(\bb{Z}/2,2) \times K(\bb{Z},4)\]
is given by identity on $\pi_2$ and multiplication-by-2 on $\pi_4$. It follows that the map of 2-connected covers $\GG/\PL\langle 2 \rangle \to \GG/\Top \langle 2 \rangle$ is equivalent to the product of the multiplication-by-2 map $2{\cdot}(-)\colon K(\bb{Z},4) \to K(\bb{Z},4)$ with the identity map of $\prod_{k \geq 2} K(\bb{Z}/2,4k+2) \times \prod_{\ell \geq 2} K(\bb{Z},4\ell)$. This leads to the following description of the map $\Top/\PL \to \GG/\PL \langle 2 \rangle$, using that the inclusion $K(\bb{Z}/2,3) \to K(\bb{Z},4)$ of the fibre of $2{\cdot}(-)$ represents $\beta \iota_3 \in H^4(K(\bb{Z}/2,3);\bb{Z})$, the Bockstein of the fundamental class.

\begin{corollary}\label{cor:map-description} The map $\Top/\PL \to \GG/\PL \langle 2\rangle$ is equivalent to the product of
	\[\beta \iota_3 \colon K(\bb{Z}/2,3) \lra K(\bb{Z},4),\]
with the inclusion of the basepoint into $\prod_{k \geq 2} K(\bb{Z}/2,4k+2) \times \prod_{\ell \geq 2} K(\bb{Z},4\ell)$.
\end{corollary}

%\begin{remark}The map $\widetilde{f}$ is in fact $\delta\Sq^1 \iota_3$, the integral lift of the first Steenrod square of the fundamental class $\iota_3 \in H^3(K(\bb{Z}/2,3);\bb{Z}/3)$. We will not need this, so forego a proof.\end{remark}

\section{Some cohomology groups of $K(\bb{Z},4)$} \label{sec:cohomology}
Given \cref{cor:map-description}, it will be crucial to understand some cohomology groups of the Eilenberg--Mac Lane space $K(\bb{Z},4)$. We will need two reduction maps
\begin{align*}\overline{(-)} &\colon H^*(X;\bb{Z}) \lra H^*(X;\bb{Z}/4) \\
	\widetilde{(-)} &\colon H^*(X;\bb{Z}/4) \lra H^*(X;\bb{Z}/2)\end{align*}
We will also use $\widetilde{(-)}$ for the composition $H^*(X;\bb{Z}) \to  H^*(X;\bb{Z}/2)$ of these maps.

\begin{lemma}\label{lem:cohomology} We have that 
	\[H^8(K(\bb{Z},4);\bb{Z}/4) \cong \bb{Z}/4 \cdot \{\overline{\iota}_4^2\} \quad \text{and} \quad H^9(K(\bb{Z},4);\bb{Z}/2) = 0.\]
\end{lemma}

\begin{proof}We start with $\bb{Z}/2$-cohomology. By \cite[Chapter 9, Theorem 3]{MosherTangora}, the cohomology ring $H^*(K(\bb{Z},4);\bb{Z}/2)$ is the polynomial ring on generators $\Sq^I(\iota_4)$ with $I = (i_1,\cdots,i_r)$ an admissible sequence of excess $e(I)<4$ and $i_r \neq 1$. This shows that 
	\begin{align*}H^7(K(\bb{Z},4);\bb{Z}/2) &= \bb{Z}/2 \cdot \{\Sq^3 \widetilde{\iota}_4\}, \\
	H^8(K(\bb{Z},4);\bb{Z}/2) &= \bb{Z}/2 \cdot \{\widetilde{\iota}_4^2\}, \\
	H^9(K(\bb{Z},4);\bb{Z}/2) &= 0\end{align*}
with $\iota_4 \in H^4(K(\bb{Z},4);\bb{Z})$ the generator and $\widetilde{\iota}_4$ its reduction modulo 2. 

To determine the $\bb{Z}/4$-cohomology in degree $8$ we use that $H^8(K(\bb{Z},4);\bb{Q}) = \bb{Q}$ generated by $\iota_4^2$. The universal coefficient theorem then tells us that there must be a $\bb{Z}/4 \subset H^8(K(\bb{Z},4);\bb{Z}/4)$, generated by an element $y$ satisfying $c \cdot y = \overline{\iota}_4^2$ for some $c \geq 1$. Working $2$-locally, we may assume that $c$ is a power of $2$. Since the Bockstein homomorphism associated to the long exact sequence $0 \to \bb{Z}/2 \to \bb{Z}/4 \to \bb{Z}/2 \to 0$ of coefficients agrees with $\Sq^1$ and hence vanishes on $H^7$ and $H^8$ using the Adem relations, the Bockstein long exact sequence gives a short exact sequence
	\[0 \lra H^8(K(\bb{Z},4);\bb{Z}/2) \lra H^8(K(\bb{Z},4);\bb{Z}/4) \lra H^8(K(\bb{Z},4);\bb{Z}/2) \lra 0\]
	with right map given by the reduction homomorphism $H^*(X;\bb{Z}/4) \to H^*(X;\bb{Z}/2)$. This shows that the cohomology group $H^8(K(\bb{Z},4);\bb{Z}/4)$ has $4$ elements, so must be equal to $\bb{Z}/4$. Moreover, the reduction modulo 2 maps its generator $y$ to $\widetilde{\iota}_4^2$ and we must have $c=1$.
\end{proof}

For an infinite loop space $X$, or more generally an $H$-space, the product map
\[\mu \colon X \times X \lra X\]
induces a map $\Delta \colon H^*(X;A) \to H^*(X \times X;A)$ on cohomology with coefficients in $A$. If $A$ is a commutative ring, then the universal coefficients theorem provides an injective map $H^*(X;A) \otimes_A H^*(X;A) \to H^*(X \times X;A)$. We say that $x \in H^*(X;A)$ is \emph{primitive} if $\Delta(x)$ is the image of $x \otimes 1 + 1 \otimes x$ under this inclusion. 

%We will need two properties of such elements:
%\begin{itemize}
%	\item If $f \colon X \to Y$ is a map of infinite loop spaces, or more generally $H$-spaces, then naturality of the K\"unneth theorem implies that $f^*$ takes primitive elements to primitive elements.
%	\item Given a fibre sequence of connective spectra 
%	\[X \lra Y \overset{f}\lra \Sigma^n HA,\]
%	the element $(\Omega^\infty f)^*(\iota_n) \in H^n(\Omega^\infty Y;A)$ is primitive. This follows from the previous point and the fact that for degree reasons, the fundamental class $\iota_n \in H^n(K(A,n);A)$ is primitive.
%\end{itemize}

\begin{lemma}\label{lem:prim-kz4} The primitive elements in $H^8(K(\bb{Z},4);\bb{Z}/4)$ are $0$ and $2 \cdot \overline{\iota}_4^2$.\end{lemma}

\begin{proof}Since $\Delta$ is a homomorphism with respect to the cup product, we have that $\Delta(\overline{\iota}_4^2) = \overline{\iota}_4^2 \otimes 1 + 2 \overline{\iota}_4 \otimes \overline{\iota}_4 + 1 \otimes \overline{\iota}_4^2$, so $\overline{\iota}_4^2$ is not primitive. However, a 2-multiple of it visibly is.\end{proof}

\section{Proof of \cref{athm:k}} \label{sec:proofs} To compute the first Postnikov $k$-invariant $k_1$, our starting point is the map of fibre sequences, up to degree 8, of infinite loop spaces
\[\begin{tikzcd} \Top/\OO \rar \dar & \Top/\PL \simeq K(\bb{Z}/2,3) \rar{\pi_1} \dar{\phi} & K(\bb{Z}/4,8) \dar[equal] \\ 
\GG/\OO\langle 2 \rangle \rar & \GG/\PL\langle 2 \rangle \rar{\pi_2} & K(\bb{Z}/4,8).\end{tikzcd}\]
Then we have $k_1 = \pi_1^*(\iota_8)$, the pullback of the identity element $\iota_8 \in H^8(K(\bb{Z}/4,8);\bb{Z}/4)$ along the right map in the top fibre sequence of infinite loop spaces. By naturality $\pi_1^*(\iota_8) = \phi^* \pi_2^*(\iota_8)$. Note that $x \coloneqq \pi_2^*(\iota_8)$ is primitive because (i) pulling back along a map of infinite loop spaces, or more generally $H$-spaces, preserves primitive elements and (ii) for degree reasons, the identity element $\iota_8 \in H^8(K(\bb{Z}/4,8);\bb{Z}/4)$ is primitive.

By \cref{cor:map-description} the map $\phi$ factors as
\[\phi \colon \Top/\PL \simeq K(\bb{Z}/2,3) \overset{\beta \iota_3}{\lra}  K(\bb{Z},4) \overset{i'}\lra \GG/\PL\langle 2 \rangle\]
so $k_1 = (\beta \iota_3)^* (i')^*(x)$. The right map $i' \colon K(\bb{Z},4) \to \GG/\PL\langle 2 \rangle$ is obtained as a connective cover of the inclusion $i \colon Y \to \GG/\PL$, which is a map of $H$-spaces by \cite[Theorem 4.34]{MadsenMilgram} (in fact, of double loop spaces by \cite[Theorem 7.25]{MadsenMilgram}), and hence so is $i'$. Hence $(i')^*(x)$ is primitive, and thus the element $k_1$ is obtained by pulling back a primitive element in $H^8(K(\bb{Z},4);\bb{Z}/4)$. In \cref{lem:prim-kz4} we saw all primitive elements in $H^8(K(\bb{Z},4);\bb{Z}/4)$ are 2-multiples, and hence must vanish upon pulling back to $H^8(K(\bb{Z}/2,3);\bb{Z}/4)$, as the latter is 2-torsion. This can be seen at least in the following two ways: (1) combine the universal coefficients theorem (including its non-natural splitting) and \cite[Table C.2]{Clement}, or (2) apply \cite[Proposition 3E.3]{Hatcher}, using the description of $H^*(K(\bb{Z}/2,3);\bb{Z}/2)$ of \cite[Chapter 9, Theorem 2]{MosherTangora} and recalling that $\Sq^1$ is the Bockstein for the short exact sequence of coefficients for $0 \to \bb{Z}/2 \to \bb{Z}/4 \to \bb{Z}/2 \to 0$ \cite[Chapter 3, Theorem 1(4)]{MosherTangora}.

\smallskip

To compute the second Postnikov $k$-invariant $k_2$, we argue similarly that it is the pullback of an element of $H^9(K(\bb{Z},4);\bb{Z}/2)$. But this group vanishes by \cref{lem:cohomology}.

\bibliographystyle{amsalpha}
\bibliography{refs}

\vspace{+0.2cm}
\end{document}